\documentclass[graybox]{svmult}

\usepackage{type1cm}        
\usepackage{makeidx}         
\usepackage{graphicx}        
\usepackage[font=small,labelfont=bf]{caption}
\usepackage{multicol}        
\usepackage[bottom]{footmisc}
\usepackage{tikz}

\usepackage{newtxtext}       %
\usepackage[varvw]{newtxmath}       

\usepackage{mathtools}
\usepackage{booktabs}
\usepackage{float}
\usepackage[table]{xcolor}

\renewcommand{\refeq}[1]{Eq.~\ref{#1}}           
\newcommand{\refthm}[1]{Thm.~\ref{#1}}           
\newcommand{\reftab}[1]{Table~\ref{#1}}          
\newcommand{\reffig}[1]{Fig.~\ref{#1}}           

\graphicspath{figures}

\makeindex             

\begin{document}

\title*{Sharpened PCG Iteration Bound for High-Contrast Heterogeneous Scalar Elliptic PDEs}
\author{Philip Soliman\orcidID{0009-0000-1061-6659} \and Filipe Cumaru\orcidID{0009-0003-9516-4226} \and Alexander Heinlein\orcidID{0000-0003-1578-8104}}
\institute{Philip Soliman \at Delft Institute of Applied Mathematics, Faculty of Electrical Engineering, Mathematics
    and Computer Science, Delft, The Netherlands. \email{philipsoliman4133@gmail.com}
    \and Filipe Cumaru \at Delft Institute of Applied Mathematics, Faculty of Electrical Engineering, Mathematics
    and Computer Science, Delft, The Netherlands. \email{f.a.cumarusilvaalves@tudelft.nl}
    \and Alexander Heinlein \at Delft Institute of Applied Mathematics, Faculty of Electrical Engineering, Mathematics
    and Computer Science, Delft, The Netherlands. \email{a.heinlein@tudelft.nl}
}
%
%

\maketitle

\abstract{A new iteration bound for the preconditioned conjugate gradient (PCG) method is presented that more accurately captures convergence for systems with clustered eigenspectra, where the classical condition number-based bound is too pessimistic. By using the edge eigenvalues of each cluster in the spectral distribution, the bound is shown to be orders of magnitude sharper than the classical bound for certain examples. Its effectiveness is demonstrated on a high-contrast elliptic PDE preconditioned with a two-level overlapping Schwarz preconditioner, where the performance of different (algebraic) coarse spaces is successfully distinguished. A key contribution of this work is the observation that, for certain high-contrast problems, simpler coarse spaces can be made competitive in terms of PCG convergence. Conversely, more complex preconditioners are not always required. Finally, it is shown that the bound can be estimated effectively from Ritz values computed during early PCG iterations.
}

\section{Introduction}\label{sec:introduction}
We study the convergence of the preconditioned conjugate gradient (PCG) method for the linear system $A\mathbf{u}=\mathbf{b}$ arising from a finite element discretization of a scalar elliptic partial differential equation (PDE) with a high-contrast heterogeneous coefficient function. In particular, let $\Omega\subset\mathbb{R}^d$ ($d=2,\,3$) be a bounded domain with Lipschitz boundary $\partial\Omega$, and $\mathcal{C}\in L^\infty(\Omega)$ be scalar field defined on $\Omega$. Let us consider the problem: find $u$ such that
\begin{equation}
    \begin{aligned}
        -\nabla\cdot\left(\mathcal{C}\nabla u\right) & = f   & \text{in } & \Omega,         \\
        u                                            & = u_D & \text{on } & \partial\Omega.
    \end{aligned}
    \label{eq:elliptic_problem}
\end{equation}
We employ a two-level overlapping additive Schwarz (2-OAS) preconditioner to the linear system $A\mathbf{u} = \mathbf{b}$ resulting from discretization of \refeq{eq:elliptic_problem}. Let $\Omega_i$, $i=1,\ldots,N$ be a non-overlapping domain decomposition of $\Omega$. Correspondingly, let $\Omega'_i \supset \Omega_i$ be subdomains with overlap $\delta$. We denote by $R_i$ the restriction operator such that $A_i = R_i A R_i^T$ is the submatrix of $A$ with local Dirichlet boundary conditions on $\Omega'_i$. Finally, let $\Phi$ be the prolongation operator whose columns span a coarse space and $A_0 = \Phi A \Phi^T$ be the Galerkin projection of $A$ onto that coarse space. The 2-OAS preconditioner is then given by
\begin{equation}
    M^{-1}_{2-\text{OAS}} = \Phi A_0^{-1} \Phi^T + \sum_{i=1}^N R_i^T A_i^{-1} R_i,
    \label{eq:two_level_oas}
\end{equation}
leading to the preconditioned system
\begin{equation}
    \tilde{A}\mathbf{u} \coloneqq M^{-1}_{2-\text{OAS}} A \mathbf{u} = M^{-1}_{2-\text{OAS}} \mathbf{b} \coloneqq \tilde{\mathbf{b}}.
    \label{eq:preconditioned_system}
\end{equation}
The PCG method's convergence and scalability depends heavily on the chosen coarse space, especially for high-contrast coefficients. This work focuses on three such coarse spaces: the algebraic multiscale solver (AMS) \cite{ams_framework_Wang2014}, generalized Dryja-Smith-Widlund (GDSW) \cite{gdsw_coarse_space_Dohrmann2008}, and the reduced-dimension GDSW (RGDSW) \cite{rgdsw_coarse_space_Dohrmann2017}.

\begin{figure}[H]
    \begin{minipage}[c]{0.38\textwidth}
        \caption{Example of a periodic high-contrast coefficient distribution with multiple inclusions along subdomain interfaces for a subdomain size $H$ = 1/4 and local problem of size $H/h = 16$. Non-overlapping subdomain boundaries $\partial\Omega_i$ and grid elements are indicated with thick and thin lines, respectively. \textbf{Light-blue}: $\mathcal{C}=1$; \textbf{red}: $\mathcal{C}=10^8$.}
        \label{fig:coefficient_distribution}
    \end{minipage}
    \begin{minipage}[c]{0.48\textwidth}
        \includegraphics{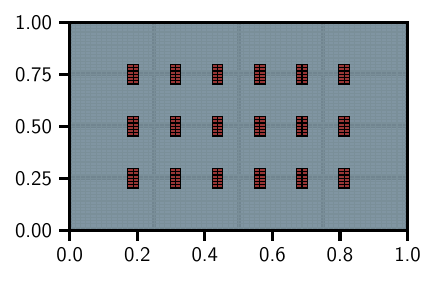}
    \end{minipage}
\end{figure}

We continue the work done in \cite{ams_coarse_space_comp_study_Alves2024}, which studied the numerical performance of the 2-OAS preconditioner with the AMS, GDSW and RGDSW coarse spaces for the high-contrast elliptic problem in \refeq{eq:elliptic_problem}. The key finding relevant to this work is that the traditional condition number-based bound fails to capture the observed differences in iteration counts between the coarse spaces, if $\mathcal{C}$ contains the regular pattern of multiple high coefficient inclusions crossing subdomain interfaces shown in \reffig{fig:coefficient_distribution}. As shown in \cite[Fig. 4]{ams_coarse_space_comp_study_Alves2024}, while AMS and GDSW significantly outperformed RGDSW in terms of iteration count, each of their condition numbers were comparable, indicating that the condition number alone is insufficient to predict the PCG iteration count. The authors suggest that these differences can be attributed to the distinct eigenvalue distributions of the preconditioned systems.

In this work, we introduce a sharper iteration bound for the PCG method based on the edge eigenvalues of each cluster in the spectral distribution. Additionally, we show that, in practice, this bound can be estimated using Ritz values resulting from early PCG iterations.

\section{Multi-cluster PCG iteration bound}\label{sec:multi_cluster_bound}
Let $\mathbf{u}_0$ be an initial guess and $\mathbf{u}_1, \dots, \mathbf{u}_m$ be subsequent PCG iterates. Also, let $\mathbf{u}^*$ be the exact solution of \refeq{eq:preconditioned_system}. Assuming $M^{-1}_{2-\text{OAS}}$ is symmetric positive definite (SPD), we can define $\tilde{A}_s = M_{2-\text{OAS}}^{-1/2}AM_{2-\text{OAS}}^{-1/2}$. Then, the $\tilde{A}_s$-norm relative error
\[
    \epsilon_m = \frac{\|\mathbf{u}^* - \mathbf{u}_m\|_{\tilde{A}_s}}{\|\mathbf{u}^* - \mathbf{u}_0\|_{\tilde{A}_s}}
\]
after $m$ PCG iterations is bounded by the constrained min-max problem
\begin{equation}
    \epsilon_m < \min_{r \in \mathcal{P}_{m}, r(0) = 1} \max_{\lambda \in \sigma} |r(\lambda)|,
    \label{eq:cg_error_bound}
\end{equation}
where $\mathcal{P}_{m}$ is the space of polynomials of degree at most $m$ and $\sigma \coloneqq \sigma(\tilde{A}) = \{\lambda_1, \ldots, \lambda_n\}$ is the spectrum of $\tilde{A}$, ordered from smallest to largest. Note that we used similarity of $\tilde{A}$ and $\tilde{A}_s$ as well as \cite[Sect. 9.2 and Thm. 6.29]{iter_method_saad} to obtain \refeq{eq:cg_error_bound}. In the classical theory of the CG method, we assume a uniformly distributed spectrum and simplify the discrete spectrum $\sigma$ in \refeq{eq:cg_error_bound} to a continuous interval $\sigma_1 \coloneqq [\lambda_1,\lambda_n]$. Correspondingly, let $m_1\in\mathbb{N}^+$ denote the classical PCG iteration bound. We can derive $m_1$ by introducing the polynomial
\begin{equation}
    r(\lambda) \leftarrow r_1(\lambda) \coloneqq \hat{C}_{m_1} \coloneqq \frac{C_{m_1}\left(T(\lambda)\right)}{C_{m_1}\left(T(0)\right)},
    \label{eq:chebyshev_polynomial}
\end{equation}
where $C_{m_1}$ is the $m_1^{\textrm{th}}$-degree Chebyshev polynomial of the first kind and
\[
    T(\lambda) \coloneqq \frac{2\lambda - (\lambda_1 + \lambda_n)}{\lambda_n - \lambda_1}.
\]
The polynomial $r_1(\lambda)$ solves the minimization problem in \refeq{eq:cg_error_bound} exactly in the interval $\sigma_1$. Finally, let some relative error tolerance $\epsilon$ be given. The classical CG iteration bound follows by requiring $\max_{\lambda \in \sigma_1} |r_1(\lambda)|<\epsilon$ and solving this inequality for the degree $m_1$, yielding
\begin{equation}
    m_1(\kappa) \coloneqq \left\lfloor \frac{\sqrt{\kappa}}{2} \ln\left(\frac{2}{\epsilon}\right) + 1 \right\rfloor,
    \label{eq:cg_iteration_bound_1_cluster}
\end{equation}
where $\kappa = \lambda_n/\lambda_1$ is the condition number of $\tilde{A}$; c.f. \cite[Eq. 6.128]{iter_method_saad}.

High-contrast coefficients lead to clustered eigenspectra with condition numbers on the order of the contrast in $\mathcal{C}$; see \cite{msfem_coarse_space_Graham_2007}. For such cases, $m_1$ is too pessimistic since it only depends on the relatively large condition number $\kappa$ and not on the full spectral distribution. To address this overestimation of $m$ by $m_1$, we consider the simplest, non-trivial case of a spectrum with a single spectral gap; see \cite[Fig. 5]{ams_coarse_space_comp_study_Alves2024} for examples of such a spectrum. Let $k$ be the \textit{partition index} corresponding to those eigenvalues that demarcate the spectral gap as the interval $[\lambda_k, \lambda_{k+1}]$. We define the two-cluster spectrum and PCG iteration bound as $\sigma \subset [\lambda_1,\lambda_k] \cup [\lambda_{k+1},\lambda_n] \coloneqq \sigma_2$ and $m_2\in\mathbb{N}^+$, respectively. Similar to \refeq{eq:chebyshev_polynomial}, we set
\begin{equation*}
    r(\lambda) \leftarrow r_2(\lambda) \coloneqq \hat{C}^{(1)}_{p}(\lambda)\hat{C}^{(2)}_{m_2-p}(\lambda),
\end{equation*}
where
\begin{equation}
    \hat{C}^{(i)}_{q}(\lambda) \coloneqq \frac{C_q\left(T_i(\lambda)\right)}{C_q\left(T_i(0)\right)} \quad q = p, m_2 - p,
    \label{eq:generalized_scaled_chebyshev_polynomial}
\end{equation}
is the scaled, $q^{\textrm{th}}$-degree Chebyshev polynomial for the $i$-th cluster with
\begin{equation}
    T_i(\lambda) \coloneqq \frac{2\lambda - (\lambda_{k_{i-1}+1} + \lambda_{k_i})}{\lambda_{k_i} - \lambda_{k_{i-1}+1}} \textrm{ with } k_0=0,\ k_1=k,\ k_2=n.
    \label{eq:Ti}
\end{equation}
Although $r_2(\lambda)$ does not necessarily solve the minimization problem in \refeq{eq:cg_error_bound} in the union of intervals, $\sigma_2$, one can still derive an expression for $m_2$ that is sharper than $m_1$ for clustered spectra. Namely,
\begin{equation}
    m_2(\kappa, \kappa_1, \kappa_2)
    \coloneqq\left\lfloor
    1
    + \frac{\sqrt{\kappa_2}}{2}\ln\left(\frac{4\kappa}{\kappa_1}\right)
    + \frac{1}{2}\ln\left(\frac{2}{\epsilon}\right)\left(
    \sqrt{\kappa_1}
    + \sqrt{\kappa_2}
    + \frac{\sqrt{\kappa_1\kappa_2}}{2}\ln\left(\frac{4\kappa}{\kappa_1}\right)
    \right)
    \right\rfloor,
    \label{eq:cg_iteration_bound_2_clusters}
\end{equation}
where $\kappa = \lambda_n/\lambda_1$, $\kappa_1 = \lambda_k/\lambda_1$ and $\kappa_2 = \lambda_n/\lambda_{k+1}$; c.f. \cite[Eq. 4.4]{cg_sharpened_convrate_Axelsson1976}.

We generalize the approach to obtain $m_2$ in \cite{cg_sharpened_convrate_Axelsson1976} to multi-cluster spectra with $s$ clusters, as shown in \reffig{fig:multiple_eigenvalue_clusters}.
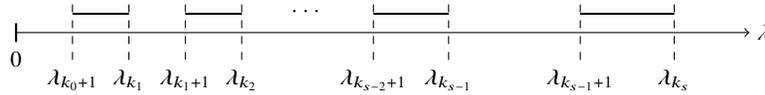
\begin{figure}[htpb!]
    \centering
    \begin{tikzpicture}[scale=2.5]

        \draw[->] (-0.1,0) -- (3.8,0) node[right] {$\lambda$};

        \foreach \x/\label in {-0.1/0}
            {
                \draw[thick] (\x,0.05) -- (\x,-0.05);
                \node[below=4pt] at (\x,0) {$\label$};
            }

        \draw[thick] (0.2,0.1) -- (0.5,0.1); 
        \draw[thick] (0.8,0.1) -- (1.1,0.1); 
        \draw[thick] (1.8,0.1) -- (2.2,0.1); 
        \draw[thick] (2.9,0.1) -- (3.4,0.1); 

        \node at (1.45,0.1) {$\cdots$};

        \foreach \x/\label in {
        0.2/{\lambda_{k_0 + 1}}, 0.5/{\lambda_{k_1}},
        0.8/{\lambda_{k_1 + 1}}, 1.1/{\lambda_{k_2}},
        1.8/{\lambda_{k_{s-2} + 1}}, 2.2/{\lambda_{k_{s-1}}},
        2.9/{\lambda_{k_{s-1} + 1}}, 3.4/{\lambda_{k_s}}
        }
        {
        \draw[thin,dashed] (\x,0.15) -- (\x,-0.15);
        \node[above] at (\x,-0.35) {$\label$};
        }

    \end{tikzpicture}
    \caption{Example of an eigenvalue spectrum consisting of multiple clusters.}
    \label{fig:multiple_eigenvalue_clusters}
\end{figure}
We introduce the partition indices $k_i, \ i = 0,\dots,s$, and define the $i^{\textrm{th}}$ cluster as the interval $I_i = [\lambda_{k_{i-1}+1}, \lambda_{k_i}]$. Then, we have $\sigma \subset \cup_{i=1}^s I_i \coloneqq \sigma_s$ and approximate the solution of \refeq{eq:cg_error_bound} by the polynomial
\begin{equation}
    r_s(\lambda) \coloneqq \prod_{i=1}^s \hat{C}^{(i)}_{p_i}(\lambda).
    \label{eq:multi_cluster_polynomial}
\end{equation}
Next, we bound \refeq{eq:multi_cluster_polynomial} by $\epsilon$ using the following \refthm{thm:chebyshev_property}, which ensures that the $i^{\textrm{th}}$ cluster's polynomial $\hat{C}^{(i)}_{p_i}(\lambda)$ does not amplify the value of $|r_s(\lambda)|$ on $I_j$ with $j<i$.
\begin{theorem}
    Let $k_i, \ i = 0,\dots,s$ be the partition indices for a spectrum with $s$ clusters $\sigma_s$, $\hat{C}^{(i)}_{p_i}(\lambda)$ be as in \eqref{eq:generalized_scaled_chebyshev_polynomial} and let the $i^{\textrm{th}}$ cluster be denoted as $I_i = [\lambda_{k_{i-1}+1},\lambda_{k_i}]$. Then, for any $i > 1$ and $1 \leq j < i$, it holds that
    \begin{equation*}
        \max_{\lambda\in I_j}\left|\hat{C}^{(i)}_{p_i}(\lambda)\right| < 1.
    \end{equation*}
    \label{thm:chebyshev_property}
\end{theorem}
\begin{proof}
    The local extrema of $C_q(x)$ and the roots of the first derivative of $C_q(x)$ for $x\in\mathbb{R}$ are given by $x_l = \cos\left(l\pi/q\right)$ with $C_q(x_l) = (-1)^l$ for $l = 1,\ldots,q-1$; c.f. \cite[Ch. 2]{mason2002chebyshev}. By the fundamental theorem of algebra, the first derivative of $C_q(x)$ has no roots other than $x_l \in (-1, 1)$ for $l=1,\ldots,q-1$. Since $T_i(\lambda) \in [-1,1]$ for $\lambda \in I_i$, all roots of the first derivative of $\hat{C}^{(i)}_{p_i}(\lambda)$ are in $I_i$. Therefore, $\hat{C}^{(i)}_{p_i}(\lambda)$ is strictly monotone outside $I_i$. Since $\left|\hat{C}^{(i)}_{p_i}(0)\right| = 1$ and $\left|\hat{C}^{(i)}_{p_i}(\lambda_{k_{i-1}+1})\right| < 1$, $\left|\hat{C}^{(i)}_{p_i}(\lambda)\right|$ is monotonically decreasing in $[0,\lambda_{k_{i-1}+1}]$. In conclusion, $\left|\hat{C}^{(i)}_{p_i}(\lambda)\right| < 1$ on $I_j$ with $I_j \subset (0,\lambda_{k_{i-1}+1}]$, which holds for $i,j$ as stated in the theorem. 
\end{proof}
A consequence of the proof of \refthm{thm:chebyshev_property} is that, for all $i,j$ such that $1 \leq j \leq i \leq s$, $\left|\hat{C}^{(j)}_{p_j}(\lambda)\right|$ is monotonically increasing in the interval $[\lambda_{k_j}, \infty)$. Therefore, the maximum value of the polynomial $r_i(\lambda)$ on an entire $i$-cluster spectrum $\sigma_i$ is attained at the largest edge of the $i$-th cluster, i.e., at $\lambda_{k_i}$. Hence, we can iteratively bound $|r_s(\lambda)|$ by requiring
\begin{equation}
    \left|\prod_{j=1}^i \hat{C}^{(j)}_{p_j}(\lambda_{k_i})\right| < \epsilon
    \label{eq:multi_cluster_bound_condition}
\end{equation}
for all $i = 1,\ldots,s$ and in that order. The iterative bounding process via \refeq{eq:multi_cluster_bound_condition} yields expressions for the degrees
\begin{equation}
    p_i = \left\lceil\log_{\gamma_i}\left(\frac{\epsilon}{2}\right) + \sum_{j=1}^{i-1} p_j \log_{\gamma_i}\left(\frac{T_j(0) + \sqrt{T_j(0)^2 -1}}{T_j(\lambda_{k_i}) - \sqrt{T_j(\lambda_{k_i})^2 -1}}\right)\right\rceil,
    \label{eq:pi}
\end{equation}
where $\gamma_i = \frac{\sqrt{\kappa_i}-1}{\sqrt{\kappa_i} + 1}$ and $\kappa_i = \frac{\lambda_{k_i}}{\lambda_{k_{i-1}+1}}$; c.f. \cite{soliman2025cg}.

Finally, the multi-cluster PCG iteration bound is given as the total degree of $r_s(\lambda)$ from \refeq{eq:multi_cluster_polynomial}
\begin{equation}
    m_s(p_1, \ldots, p_s) \coloneqq \sum_{i=1}^{s} p_i,
    \label{eq:cg_iteration_bound_s_clusters}
\end{equation}
with $p_i$ as in \refeq{eq:pi}.

Calculating $m_s$ from \refeq{eq:cg_iteration_bound_s_clusters} requires partition indices $k_i$ that split the eigenvalues into disjoint clusters. To obtain $k_i$, we define a candidate two-cluster split using the largest relative gap
\begin{equation}
    k^*(\sigma) \coloneqq \arg\max_{i} \frac{\lambda_{i+1}}{\lambda_i},
    \label{eq:partition_index_2_clusters}
\end{equation}
Next, we require the two-cluster bound from \refeq{eq:cg_iteration_bound_2_clusters} resulting from this candidate split to be smaller than the classical bound from \refeq{eq:cg_iteration_bound_1_cluster}, i.e. $m_2 < m_1$. In doing so, we derive a condition number threshold via the Lambert $W$ function, the inverse of $y\mapsto y\mathrm{e}^y$. Our analysis uses the principal negative branch $W_{-1}$ and its asymptotic expansion for $x\to0^-$; see \cite[Eq. 4.19]{evaluation_of_the_lambert_w_function_Corless1996}. Letting $\kappa_1=\lambda_{k^*}/\lambda_1$, $\kappa_2=\lambda_n/\lambda_{k^*+1}$, $x=-\left(2\sqrt{\kappa_2}\exp\left(\frac{1}{\sqrt{\kappa_2}}\right)\right)^{-1}$, $L=\ln(-x)$ and $l=\ln(-L)$ gives the threshold explicitly as
\begin{equation}
    \kappa > 4\kappa_1\kappa_2 \left(L - l + \frac{l}{L}\right)^2 + \mathcal{O}\left(\frac{\kappa_1\kappa_2l^4}{L^4}\right),
    \label{eq:threshold_inequality_explicit_expansion}
\end{equation}
If the condition in \refeq{eq:threshold_inequality_explicit_expansion} holds, we accept the candidate split and apply a greedy recursion on each subcluster until no more splits are accepted or clusters consist of singletons; see \cite[Sect. 4.2]{soliman2025cg} for the full derivation of \refeq{eq:threshold_inequality_explicit_expansion}, \cite[Alg. 8]{soliman2025cg} for the partitioning algorithm and \cite[Alg 9]{soliman2025cg} for the complete computation of $m_s$.

\section{Numerical experiments}\label{sec:experiments}
We solve the system from \refeq{eq:preconditioned_system} for the coefficient distribution shown in \reffig{fig:coefficient_distribution} using the PCG method with the $M_{2-\textrm{OAS}}$ preconditioner from \refeq{eq:two_level_oas} with the AMS, GDSW, and RGDSW coarse spaces and for subdomain sizes $H=1/4, 1/8, 1/16, 1/32, 1/64$, grid size $h = H/16$, overlap $\delta = 2h$ and relative error tolerance $\epsilon = 10^{-8}$.

First, we compute our multi-cluster bound $m_s$ using Ritz values obtained at convergence. \reffig{fig:absolute_performance} shows that the classical bound $m_1$ severely overestimates the iteration count and fails to differentiate between the coarse spaces. In contrast, $m_s$ is consistently accurate, staying within a factor of 1-10 of the true iteration count $m$ and correctly discerning the coarse spaces' performance.
\begin{figure}[htpb!]
    \centering
    \includegraphics{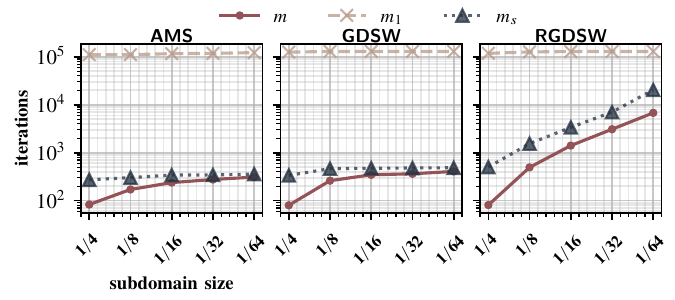}
    \caption{
        Actual PCG iterations $m$ versus the classical ($m_1$) and multi-cluster ($m_s$) bounds computed at convergence for varying subdomain sizes $H$.}
    \label{fig:absolute_performance}
\end{figure}

In \reftab{tab:cg_iteration_bounds}, we evaluate $m_s$ as an early estimator by computing it during the initial PCG iterations, i.e., within $i_{\max}$ iterations or a fraction $r$ of the total iteration count $m$. We calculate $m_s$ only after the edge Ritz values in each cluster have stabilized, checked every $\eta=5$ iterations against a tolerance $\tau=0.1$
\begin{equation}
    \frac{\lambda_{k^{(i)}_j}}{\lambda_{k^{(i-\eta)}_j}} < 1 + \tau, \quad j = 0, \ldots, s.
    \label{eq:convergence_of_extremal_ritz}
\end{equation}

The results in \reftab{tab:cg_iteration_bounds} show that $m_s$, computed using early Ritz values, provides a good estimate of the final iteration count $m$ for AMS and GDSW for $i < \min\{100, 0.5m\}$. Although the performance of $m_s$ as an upper bound deteriorates slightly for smaller subdomain sizes $H\leq1/16$ for GDSW and even more so for RGDSW, it is still able to reflect the differences in convergence speed between the three coarse spaces. In contrast, the classical bound $m_1$ is not able to do so, and again, largely overestimates $m$ in all cases. The underestimation of the number of iterations for RGDSW by $m_s$ is due to the fact that the early Ritz values do not approximate the full spectrum of the preconditioned system matrix well in these cases, as shown in \cite[Sec. 6.2]{soliman2025cg}.

\begin{table}[htpb!]
    \centering
    \caption{The number of PCG iterations required to achieve convergence $m$, corresponding iteration bounds $m_1$, $m_s$ and the iteration at which the bounds are computed $i$. Cell colors indicate if bounds are larger (blue) or smaller (red) than $m$, with shading proportional to absolute difference. Bounds are calculated with $\eta=5$, $\tau=0.1$, $i_{\max}$=100 (AMS), 100 (GDSW), 300 (RGDSW) and fraction of total run $r=0.5$.}
    \label{tab:cg_iteration_bounds}
    \begin{tabular}{lrrrrrrrrr}
        \toprule
              & \multicolumn{3}{c}{$\mathbf{H=1/4}$}                   & \multicolumn{3}{c}{$\mathbf{H=1/16}$}                   & \multicolumn{3}{c}{$\mathbf{H=1/64}$}                                                                                                                                                                                                                                                                                                                                                                              \\
              & AMS                                                    & GDSW                                                    & RGDSW                                                   & AMS                                                     & GDSW                                                    & RGDSW                                                   & AMS                                                     & GDSW                                                    & RGDSW                                                  \\
        \midrule
        $m$   & 83                                                     & 80                                                      & 81                                                      & 238                                                     & 346                                                     & 1,406                                                   & 310                                                     & 407                                                     & 6,766                                                  \\
        \cline{1-10}
        $m_1$ & {\cellcolor[HTML]{AFC9F6}} \color[HTML]{000000} 98,421 & {\cellcolor[HTML]{AFC9F6}} \color[HTML]{000000} 124,727 & {\cellcolor[HTML]{AFC9F6}} \color[HTML]{000000} 117,699 & {\cellcolor[HTML]{AFC9F6}} \color[HTML]{000000} 110,969 & {\cellcolor[HTML]{7EAFF1}} \color[HTML]{000000} 123,872 & {\cellcolor[HTML]{AFC9F6}} \color[HTML]{000000} 122,195 & {\cellcolor[HTML]{AFC9F6}} \color[HTML]{000000} 114,629 & {\cellcolor[HTML]{7EAFF1}} \color[HTML]{000000} 121,215 & {\cellcolor[HTML]{7EAFF1}} \color[HTML]{000000} 69,645 \\
        \cline{1-10}
        $m_s$ & {\cellcolor[HTML]{7EAFF1}} \color[HTML]{000000} 191    & {\cellcolor[HTML]{7EAFF1}} \color[HTML]{000000} 88      & {\cellcolor[HTML]{7EAFF1}} \color[HTML]{000000} 133     & {\cellcolor[HTML]{7EAFF1}} \color[HTML]{000000} 310     & {\cellcolor[HTML]{945357}} \color[HTML]{F1F1F1} 306     & {\cellcolor[HTML]{7EAFF1}} \color[HTML]{000000} 2,185   & {\cellcolor[HTML]{7EAFF1}} \color[HTML]{000000} 324     & {\cellcolor[HTML]{945357}} \color[HTML]{F1F1F1} 312     & {\cellcolor[HTML]{945357}} \color[HTML]{F1F1F1} 2,612  \\
        \cline{1-10}
        $i$   & 41                                                     & 36                                                      & 36                                                      & 81                                                      & 76                                                      & 241                                                     & 91                                                      & 81                                                      & 291                                                    \\
        \cline{1-10}
        \bottomrule
    \end{tabular}
\end{table}

Allowing a modest increase in PCG iterations before evaluating $m_s$ improves the early estimate for the GDSW coarse space. Indeed, in \cite[Tab.~6.3]{soliman2025cg} $i_{\max}$ was raised to 300, and produced $m_s=485$ for $H=1/64$ at iteration $i=186$, a valid upper bound for $m=407$. Similar behavior is suggested for RGDSW, although Ritz values require substantially more iterations to stabilize and yield reliable estimates.

\section{Conclusion}\label{sec:conclusion}

We introduced a new multi-cluster PCG iteration bound that is sharper for systems with clustered spectra compared to the classical bound. We demonstrated the effectiveness of the new bound on a high-contrast problem, where it correctly discerned between the relative performance of several algebraic coarse spaces in a 2-OAS preconditioner. The new bound can be estimated from early PCG iterations, although its accuracy relies on the convergence rate of the underlying Ritz values. Our results suggest that the PCG iterations are able to handle bad eigenmodes that are not resolved by the coarse space as long as those are grouped into clusters. We may thus choose less computationally complex coarse spaces in those cases. Future work could focus on refining the spectral partitioning strategy, applying the bound to more complex PDE systems, and, if possible, developing \textit{a priori} cluster edge eigenvalue estimates to remove the dependency on Ritz values.

\bibliographystyle{spmpsci}
\bibliography{bibliography}

\end{document}